\documentclass[12pt]{amsart}
\usepackage{amssymb}
\usepackage{amsmath, amscd}
\usepackage{amsthm}
\usepackage{amsmath}
\usepackage{comment}
\usepackage[table,xcdraw]{xcolor}
\usepackage{ulem}
\usepackage{tikz}
\usepackage{float}
\usepackage{graphicx}
\usepackage{circuitikz}
\usetikzlibrary{positioning}
\usepackage[colorlinks=true, linkcolor=blue,urlcolor=blue]{hyperref}

\newtheorem{theorem}{Theorem}[section]

\newtheorem{proposition}[theorem]{Proposition}
\newtheorem{lemma}[theorem]{Lemma}

\newtheorem{corollary}[theorem]{Corollary}
\theoremstyle{definition}

\newtheorem{example}[theorem]{Example}
\newtheorem{definition}[theorem]{Definition}

\newtheorem{remark}[theorem]{Remark}

\newtheorem{problem}[theorem]{Problem}

\begin{document}

\author[O. Hasanzadeh]{Omid Hasanzadeh}
\address{Department of Mathematics, Tarbiat Modares University, 14115-111 Tehran Jalal AleAhmad Nasr, Iran}
\email{o.hasanzade@modares.ac.ir; hasanzadeomiid@gmail.com}
\author[A. Moussavi]{Ahmad Moussavi$^*$}
\address{Department of Mathematics, Tarbiat Modares University, 14115-111 Tehran Jalal AleAhmad Nasr, Iran}
\email{moussavi.a@modares.ac.ir; moussavi.a@gmail.com}
\author[Peter Danchev]{Peter Danchev}
\address{Institute of Mathematics and Informatics, Bulgarian Academy of Sciences, 1113 Sofia, Bulgaria}
\email{danchev@math.bas.bg; pvdanchev@yahoo.com}

\thanks{$^*$Corresponding author: Ahmad Moussavi, email: moussavi.a@modares.ac.ir; moussavi.a@gmail.com}

\title[2-$\Delta$U rings]{Rings with 2-$\Delta$U property}
\keywords{2-$\Delta$U ring, $\Delta$U ring, $\Delta (R)$, Matrix ring}
\subjclass[2010]{16S34, 16U60}

\maketitle




\begin{abstract}
Rings in which the square of each unit lies in  $1+\Delta(R)$,   are said to be $2$-$\Delta U$, where $J(R)\subseteq\Delta(R) =: \{r \in R | r + U(R) \subseteq U(R)\}$.
The set $\Delta (R)$ is the largest Jacobson radical subring of $R$ which is closed with respect
to multiplication by units of $R$ and is studied in \cite{2}. The class of $2$-$\Delta U$ rings consists several rings including $UJ$-rings, $2$-$UJ$ rings and  $\Delta U$-rings, and we observe that $\Delta U$-rings are $UUC$. The structure of $2$-$\Delta U$ rings is studied under various conditions. Moreover, the $2$-$\Delta U$ property is studied under some algebraic constructions. 	

\end{abstract}

\section{Introduction and Basic Concepts}

In the current paper, let $R$ denote an associative ring with identity element, not necessarily commutative. Typically, for such a ring $R$, the sets $U(R)$, $Nil(R)$, $C(R)$,  $Id(R)$ and $J(R)$ represent the set of invertible elements, the set of nilpotent elements, the set of central elements,  the set of idempotent elements in $R$ and the Jacobson radical of $R$ respectively. Additionally, the ring of \( n \times n \) matrices over \( R \) and the ring of \( n \times n \) upper triangular matrices over \( R \) are denoted by \( {\rm M}_n(R) \) and \( {\rm T}_n(R) \), respectively. Traditionally, a ring is termed {\it abelian} if each idempotent element is central, meaning that $Id(R) \subseteq C(R)$.

The main material of the present study is the set $\Delta(R)$ was handled by Lam \cite [Exercise 4.24]{23} and recently, studied by  Leroy-Matczuk \cite{2}. As pointed out by the authors in \cite [Theorem 3 and 6]{2}, $\Delta(R)$ is the largest Jacobson radical subring of $R$ which is closed with respect to multiplication by all units (quasi-invertible elements) of $R$. Also, $J(R) \subseteq \Delta(R)$, and  $\Delta(R)=J(T)$, where $T$ is the subring of $R$ generated by units of $R$, and the equality $\Delta(R)=J(R)$ holds if, and only if, $\Delta(R)$ is an ideal of $R$.

It is well-known that $1+J(R) \subseteq U(R)$. A ring $R$ is said to be an $UJ$-ring if the reverse inclusion holds, i.e. $U(R)=1+J(R)$ \cite{14}. By \cite{1}, a ring $R$ is said to be 2-$UJ$ if for each $u\in U(R)$, $u^2=1+j$ where $j\in J(R)$.  These rings are generalization of $UJ$ rings. They showed that for 2-$UJ$ rings the notions semi-regular, exchange and clean rings are equivalent. Recall that a ring $R$ is called an $UU$-ring if $U(R)=1+Nil(R)$ (see \cite{12}). As a generalization of $UU$ rings Sheibani and Chen in \cite{10} introduced 2-$UU$ rings. A ring $R$ is called 2-$UU$ if the square of every unit is the sum of $1_{R}$ and a nilpotent. They showed that $R$ is strongly 2-nil-clean if, and only if, $R$ is an exchange 2-$UU$ ring. A ring $R$ is said to be regualr (unit-regular) in the sense of von Neumann if for every $a\in R$, there is an $x\in R$ (resp., $x\in U(R)$) such that $axa=a$, and $R$ is said to be strongly regular if for each $a\in R$, $a\in a^2R$. Recall that a ring $R$ is exchange if for each $a\in R$, there exist $e^2=e\in aR$ such that $1-e\in (1-a)R$, and a ring $R$ is clean if every element of $R$ is a sum of an idempotent and an unit \cite{8}. Notice that every clean ring is exchange, but the converse is not true in general, whereas the converse is true in the abelian case (see \cite[Proposition 1.8]{8}). A ring $R$ is called semi-regular if $R/J(R)$ is regular and idempotents lift modulo $J(R)$. Semi-regular rings are exchange, but the converse is not true in general (see \cite{8}). A ring $R$ is Boolean if every element of $R$ is idempotent.

 According to Chen \cite {16}, an element of a ring is called $J$-clean provided that it can be written as the sum of an idempotent and an element from its Jacobson radical. A ring is $J$-clean in the case where each of its element is $J$-clean or equivalently $R/J(R)$ is Boolean and idempotents lift modulo $J(R)$ (which is also called semi-boolean in \cite{25}). It was shown in \cite{14} that a ring $R$ is $J$-clean if, and only if, $R$ is a clean $UJ$ ring. In 2019, Fatih Karabacak et al. introduced new rings that are a generalization of $UJ$ rings. They called these rings $\Delta U$ \cite{7}. A ring $R$ is said to be $\Delta U$ if $1+\Delta(R)=U(R)$. Due to Fatih Karabacak et al. \cite{7}, a ring $R$ is called $\Delta$-clean if every element of $R$ is a sum of an idempotent and an element from the $\Delta(R)$. $\Delta$-clean rings are clean, but the converse is not true in general. They showed in \cite{7} that a ring $R$ is $\Delta$-clean if, and only if, $R$ is a clean $\Delta U$ ring. 

As a generalization of the above concepts, we introduce  2-$\Delta U$ rings. A ring $R$ is called 2-$\Delta U$ if the square of each unit is a sum of an idempotent and an element from the $\Delta(R)$ (equivalently, for each $u\in U(R)$, $u^2=1+r$ where $r\in \Delta(R)$). Clearly all $\Delta U$ rings and hence unit uniquely  clean rings  and rings with only two units are 2-$\Delta U$. Also,  2-$UJ$ rings and hence $UJ$ rings are $2$-$\Delta U$ but, the converse is not true in general. Our motivating tool is to give a satisfactory description of 2-$\Delta U$ rings by comparing their crucial properties with these of 2-$UU$ and 2-$UJ$ rings, respectively, as well as to find some new exotic properties of 2-$\Delta U$ rings that are not too characteristically frequently seen in the existing literature. 

We are now planning to give a brief program of our results established in the sequel: In section 2, we achieve to exhibit some major properties and characterizations of $\Delta U$ rings in various different aspects (see, for instance, Propositions \ref{2.1}, \ref{2.2} and \ref{2.3} and Theorem \ref{2.4}). In section 3, we establish some fundamental characterization properties of 2-$\Delta U$ rings that are mainly stated and proved in Theorems \ref{3.5}, \ref{3.13}, \ref{3.14} and \ref{3.16} and the other statments associated with them. In section 4, we give some extensions of 2-$\Delta U$ rings, for instance, polynomial extensions, matrix extensions, trivial extensions and  Morita contexts. We close our work in the final section with challenging questions, namely Problems \ref{5.1}, \ref{5.2}, \ref{5.3} and \ref{5.4}.

Now, we have the following diagram which violates the relationships between the defined above classes of rings:

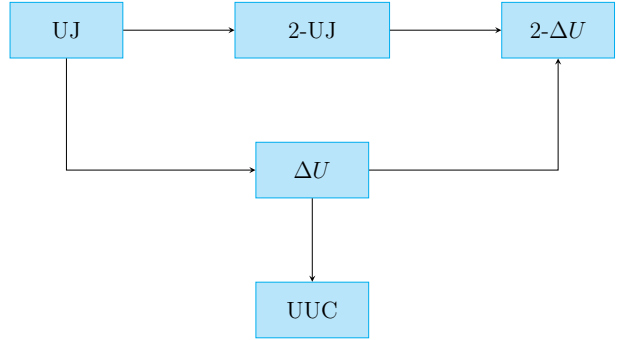
\begin{center}
	\resizebox{0.55\textwidth}{!}{ 
		\begin{tikzpicture}[node distance=1.5cm and 2cm]
			\node[draw=cyan,fill=cyan!25,minimum width=2cm,minimum height=1cm,text width=1.75cm,align=center]  (a) {UJ};
			\node[draw=cyan,fill=cyan!25,minimum width=2cm,minimum height=1cm,text width=2.5cm,align=center,right=of a](b){2-UJ};
			\node[draw=cyan,fill=cyan!25,minimum width=2cm,minimum height=1cm,text width=1.75cm,align=center,below=of b](c){$\Delta U$};
			\node[draw=cyan,fill=cyan!25,minimum width=2cm,minimum height=1cm,text width=1.75cm,align=center,right=of b](d){2-$\Delta U$};
			\node[draw=cyan,fill=cyan!25,minimum width=2cm,minimum height=1cm,text width=1.75cm,align=center,below=of c]  (e) {UUC};
			\draw[-stealth] (a) -- (b);
			\draw[-stealth] (a) |- (c);
			\draw[-stealth] (b) -- (d);
			\draw[-stealth] (c) -- (e);
			\draw[-stealth] (c) -| (d);
		\end{tikzpicture}
	}
\end{center}

\section{$\Delta$U rings} 

In this section, we investigate some properties of $\Delta U$ rings. We give relations between $\Delta U$ rings and some related rings.

\begin{definition}[\cite{7}]
A ring $R$ is called $\Delta U$  if $1+\Delta(R)=U(R)$. 	
\end{definition}

By Calugareanu and  Zhou, a ring $R$ is called $UUC$ if every unit is uniquely clean \cite {13}. 
\begin{proposition}\label{2.1}
Let \(R\) be a $\Delta$U ring. Then:
\begin{enumerate}
\item 
\(U(R) + U(R) \subseteq \Delta(R)\).
\item 
\(R\) is a \(UUC\) ring.
\item 
\((U(R) + U(R)) \cap Id(R) = 0\).
\end{enumerate}
\end{proposition}
\begin{proof}
\begin{enumerate}
\item 
Let \(x \in U(R) + U(R)\). Then \(x = u_1 + u_2\) where, \(u_1, u_2 \in U(R) = 1 + \Delta(R)\), so \(x = 1+r_1 + 1+r_2 = 2 + (r_1 + r_2)\) where, \(r_1, r_2 \in \Delta(R)\). On the other hand, we know that $2 \in \Delta(R)$ by \cite [Proposition 2.4]{7}. Also, $\Delta(R)$ is a subring of $R$. Then, \(x \in \Delta(R)\).
\item 
Assume that \(u = e + v\) where, \(u, v \in U(R)\) and \(e \in Id(R)\). It suffices to show that \(e = 0\). As \(R\) is \(\Delta U\), \(u = 1 + r\) and \(v = 1 + r^\prime\) where, \(r, r^\prime \in \Delta(R)\). Hence, \(e = 0\) by \cite [Proposition 15]{2}.
\item 
This is clear by (i) and (ii).
\end{enumerate}
\end{proof}

\begin{proposition}\label{2.2}
Let \( R \) be a $\Delta$U ring and let \( \bar{R} = R/J(R) \). The following hold:
\begin{enumerate}
\item 
For any \( u_1, u_2 \in U(R) \), \( u_1 + u_2 \neq 1 \).
\item 
For any \( \bar{u_1}, \bar{u_2} \in U(\bar{R}) \), \( \bar{u_1} + \bar{u_2} \neq \bar{1} \).
\end{enumerate}
\end{proposition}
\begin{proof}
\begin{enumerate}
\item 
This is clear by Proposition \ref{2.1} and \cite [Example 2.2]{13}.
\item 
Since \( R \) is $\Delta$U, \( \bar{R}  \) is \( \Delta U \) and hence it is \( UUC \), then the result follows from \cite [Example 2.2]{13}.
\end{enumerate}
\end{proof}

A ring $R$ is called semi-potent if every one-sided ideal not contained in $J(R)$ contains a non-zero idempotent. Moreover, a semi-potent ring $R$ is called potent if idempotents lift modulo $J(R)$.

\begin{proposition}\label{2.3}
Let \( R \) be a potent $\Delta$U ring, and \( \bar{R} = R/J(R) \). Then, we have:
\begin{enumerate}
\item 
For any \( \bar{e} = \bar{e}^2 \in \bar{R} \) and any \( \bar{u_1}, \bar{u_2} \in U(\bar{e} \bar{R} \bar{e}) \), \( \bar{u_1} + \bar{u_2} \neq \bar{e} \).
\item 
There does not exist \( \bar{e} = \bar{e}^2 \in \bar{R} \) such that \( \bar{e} \bar{R} \bar{e} \cong M_2(S) \) for some ring \( S \).
\end{enumerate}
\end{proposition}
\begin{proof}
\begin{enumerate}
\item 
Given \( \bar{e}, \bar{u_1}, \bar{u_2} \) as in (i), we can assume that \( e^2 = e \in R \), because idempotents lift modulo \( J(R) \). Then, \( \bar{e} \bar{R} \bar{e} \cong eRe/J(eRe) \). Finally, since \( e R e \) is $\Delta$U by \cite [Proposition 2.6]{7}, (i) follows directly by Proposition \ref{2.2}(i).
\item 
Note that, in a \( 2 \times 2 \) matrix ring, it is always true that
		\[
		\begin{pmatrix}
			1 & 0 \\
			0 & 1 
		\end{pmatrix} 
		= \begin{pmatrix}
			1 & 1 \\
			1 & 0
		\end{pmatrix} + \begin{pmatrix}
			0 & -1 \\
			-1 & 0
		\end{pmatrix} \in U(M_2(S)) + U(M_2(S)).
		\]
Hence, there exist \( \bar{u_1}, \bar{u_2} \in U(\bar{e} \bar{R} \bar{e}) \) such that \( \bar{u_1} + \bar{u_2} = \bar{e} \). This is a contradiction with (i).
\end{enumerate}
\end{proof}

A ring $R$ is called reduced if it contains no non-zero nilpotent elements.

\begin{theorem}\label{2.4}
Let \( R \) be a semi-potent ring. The following statements are equivalent:
\begin{enumerate}
\item 
 \( R \) is a $\Delta$U ring.
\item 
 \( R/J(R) \) is  Boolean.
\item 
 \( R \) is a \( UJ \) ring.
\item 
 $R/J(R)$ is a $UU$ ring.
\end{enumerate}
\end{theorem}
\begin{proof}
\( (i) \Rightarrow (ii) \) Since \( R \) is semi-potent, \( R/J(R) \) is semi-potent (indeed, potent). Also, \( R/J(R) \) is $\Delta$U. So without loss of generality it can be assumed that \( J(R) = 0 \). Then by \cite[Theorem 4.4]{7}, \( R \) is reduced and hence it is abelian. Now, assume that there exists \( x \in R \) such that \( x - x^2 \neq 0 \) in \( R \). Since \( R \) is a semi-potent ring, there exists \( e = e^2 \in R \) such that \( e \in (x - x^2)R \). So, write \( e = (x - x^2)y \) for some \( y \in R \). Since \( e \) is central, \( [e r (1-e)]^2 = 0 = [(1-e) re]^2 \), we have \( er(1-e) = 0 = e (1-e) re \). We can write:
	\[
	e = e x. e (1-e). ey,
	\]
so that both \( ex, e(1-x) \in U(eRe) \). But, we know that \( eRe \) is a $\Delta$U ring. However, \( e x + e (1-x) = e \), which contradicts Proposition \ref{2.2}(i). Therefore, \( R \) is a Boolean ring, as desired.
	
\( (ii) \Rightarrow (iii) \) Let us assume \( u \in U(R) \). Then, \( \bar{u} \in U(\bar{R} = R/J(R)) \). Since \( \bar{R} \) is a Boolean ring, we have \( \bar{u} = \bar{1} \), which implies \( u-1 \in J(R) \).
	
\( (iii) \Rightarrow (i) \) This is clear, because always we have \( J(R) \subseteq \Delta(R) \).

\( (ii) \Rightarrow (iv) \) This is obvious.

\( (iv) \Rightarrow (ii) \) We know that $R/J(R)$ is semi-potent. Then, it is strongly nil-clean by \cite[Theorem 2.25]{11}. Hence $R/J(R)$ is an exchange $UU$ ring by \cite[Theorem 4.3]{12}. Thus, it is Boolean  by \cite[Theorem 4.1]{12}.
\end{proof}

\begin{corollary}\label{2.5}
A regular ring \( R \) is $\Delta$U if, and only if, \( R \) is \( UJ \) if, and only if, \( R \) is \( UU \) if, and only if, \( R \) is Boolean.
\end{corollary}
\begin{proof}
Since $R$ is regular, $J(R)=0$ and $R$ is semi-potent. So the result follows from Theorem \ref{2.4}.	
\end{proof}

\begin{corollary}\label{2.6}
Let \( R \) be a potent ring. Then, the following are equivalent:
\begin{enumerate}
\item 
\( R \) is a $\Delta$U ring.
\item 
\( R/J(R) \) is a $\Delta$U ring.
\item 
\( R/J(R)\) is a Boolean ring.
\item 
\( R \) is a \( UJ \) ring.
\item 
\( R/J(R) \) is a \( UJ \) ring.
\item 
\( R/J(R) \) is a \( UU \) ring.
\end{enumerate}
\end{corollary}
\begin{proof}
We know that every potent ring is semi-potent. Then, (i), (iii), (iv) and (vi) are equivalent by Theorem \ref{2.4}. On the other hand (i) and (ii) are equivalent by \cite[Proposition 2.4]{7}. Also, (iv) and (v) are equivalent by \cite[Proposition 1.3]{14}.	
\end{proof}	

\begin{corollary}\label{2.7}
Let \( R \) be an Artinian ring. Then, the following are equivalent:
\begin{enumerate}
\item 
\( R \) is a $\Delta$U ring.
\item 
\( R \) is a \( UJ \) ring.
\item 
\( R \) is a \( UU \) ring.
\end{enumerate}
\end{corollary}
\begin{proof}
We know that every Artinian ring is always clean. Also, since \( R \) is Artinian, we have \( J(R) \subseteq Nil(R) \).
\end{proof}

\begin{corollary}\label{2.8}
Let \( R \) be a finite ring. Then, the following conditions are equivalent:
\begin{enumerate}
\item 
\( R \) is a $\Delta$U ring.
\item 
\( R \) is a \( UJ \) ring.
\item 
\( R \) is a \( UU \) ring.
\end{enumerate}
\end{corollary}
\begin{proof}
In fact, any finite ring is  Artinian.
\end{proof}

\begin{corollary}\label{2.9}
For a ring \( R \), the following two conditions are equivalent:
\begin{enumerate}
\item 
\( R \) is a potent $\Delta$U ring.
\item 
\( R \) is a \( J \)-clean ring.
\end{enumerate}
\end{corollary}
\begin{proof}
\((ii) \Rightarrow (i)\) This is clear by \cite [Theorem 3.2]{14} and Corollary \ref{2.6}.
	
\((i) \Rightarrow (ii)\) Applying Corollary \ref {2.6}, we know that \( R/J(R) \) is Boolean. Therefore, for each \( a \in R \), we have \( a - a^2 \in J(R) \). Since \( R \) is a potent ring, there exists an idempotent \( e \in R \) such that \( a - e \in J(R) \). Thus, \( R \) is a \( J \)-clean ring, as wanted.
\end{proof}

Let $Nil_{*}(R)$ denote the prime radical (or lower nil-radical) of a ring $R$, i.e. the intersection of all prime ideals of $R$. We know that $Nil_{*}(R)$ is a nil-ideal of $R$. It is long known that a ring $R$ is called {\it $2$-primal} if its lower nil-radical $Nil_{*}(R)$ consists precisely of all the nilpotent elements of $R$. For instance, it is well known that both reduced rings and commutative rings are both $2$-primal. For an endomorphism $\alpha$ of a ring $R$, $R$ is called {\it $\alpha$-compatible} if, for any $a,b\in R$, $ab=0\Longleftrightarrow a\alpha (b)=0$, and in this case $\alpha$ is clearly injective.

Let $R$ be a ring and $\alpha : R \to R$ is a ring endomorphism and $R[x; \alpha]$ denotes the ring of skew polynomial over $R$, with multiplication defined by $xr = \alpha(r)x$ for all $r \in R$. In particular, $R[x] = R[x; 1_R]$ is the ring of polynomial over $R$.

\begin{proposition}\label{2.10}
Let \( R \) be a \(2\)-primal ring and \( \alpha \) be an endomorphism of \( R \). If \( R \) is \( \alpha \)-compatible, then
	\[
	\Delta(R[x; \alpha]) = \Delta(R) + J(R[x; \alpha]).
	\]
\end{proposition}
\begin{proof}
Suppose first that $R$ is a reduced ring. As $R$ is \( \alpha \)-compatible, by \cite[Corollary 2.12]{15} we have $U(R[x; \alpha])=U(R)$. Also, it is easy to see that, $\Delta(R)\subseteq \Delta(R[x; \alpha])$. Let $r+r_0\in \Delta(R[x; \alpha])$, where $r\in R[x; \alpha]x$ and $r_0\in R$. Then, for any $u\in U(R)$, $r+r_0+u\in U(R)$. This shows that $r=0$ and $r_0+u\in U(R)$. Thus, we conclude that $\Delta(R[x; \alpha])\subseteq \Delta(R)$ and hence $\Delta(R[x; \alpha])=\Delta(R)$. Now assume that $R$ is $2$-primal. Clearly $Nil_{*}(R[x; \alpha])=Nil_{*}(R)[x; \alpha]\subseteq J(R[x; \alpha])$ by \cite[Lemma 2.2]{15}. As $R$ is $2$-primal, $R/Nil_{*}(R)$ is reduced, so we have \(J(R[x; \alpha]) = \operatorname{Nil}_*(R[x; \alpha]) = \operatorname{Nil}_*(R)[x; \alpha]\). By the first part of the proof applied to $R/Nil_{*}(R)$ and using \cite [Proposition 6(3)]{2}, we have: 
	\[
\Delta(R)+Nil_{*}(R)[x; \alpha]=\Delta(\dfrac {R}{Nil_{*}(R)}[x; \alpha])=\Delta(\dfrac {R[x; \alpha]}{J(R[x; \alpha])})=\dfrac {\Delta(R[x; \alpha])}{J(R[x; \alpha])}
.\]
By the above we conclude  the desired equality.
\end{proof}

For any two elements $ a,b \in R$, $1-ab$ is an unit if, and only if, $1-ba$ is an unit. This result is known as Jacobson’s lemma for units. There are several analogous results in the literature. We have the following.

\begin{corollary}\label{2.11}
Let \( R \) be a $\Delta$U ring and let \( a, b \in R \). Then, \( 1-ab \in \Delta(R) \) if, and only if, \( 1-ba \in \Delta(R) \).
\end{corollary}
\begin{proof}
Assuming that \( 1-ab \in \Delta(R) \), we have \( ab \in U(R) \). Therefore, by \cite [Proposition 2.4]{7}, \( a \in U(R) \). Thus,
	\[
	1-ba = a^{-1}(1-ab)a \in \Delta(R),
	\]
because \(\Delta(R)\) is closed with respect to multiplication by all units (see \cite [Theorem 3]{2}). The converse is similar.
\end{proof}

\section{2-$\Delta$U rings}

In this section, we introduce 2-$\Delta$U rings and investigate its elementary properties.
We now give our main definition.

\begin{definition}
A ring $R$ is called 2-$\Delta$U if the square of each unit is a sum of an idempotent and an element from the $\Delta(R)$ (equivalently, for each $u\in U(R)$, $u^2=1+r$ where $r\in \Delta(R)$).	
\end{definition}

\begin{example}\label{3.31}
Clearly, 2-$UJ$ rings are 2-$\Delta$U. But, the converse is not true in general. For example, consider the ring \( R = F_2\langle x, y \rangle / \langle x^2 \rangle \). Then \( J(R) = 0 \), \( \Delta(R) = F_2 x + xRx \) and \( U(R) = 1 + F_2x + xRx \). Then \( R \) is $\Delta$U by \cite [Example 2.2]{7} and hence it is 2-$\Delta$U. But clearly \( R \) is not 2-$UJ$.
\end{example}

\begin{example}\label{3.32}
The ring $\mathbb{Z}_3$ is 2-$\Delta U$, but is not $\Delta U$.	
\end{example}

\begin{proposition}\label{3.1}
A direct product \(\prod_{i \in I} R_i\) of rings is 2-$\Delta$U if, and only if, each \(R_i\) is 2-$\Delta$U.
\end{proposition}
\begin{proof}
As \(\Delta(\prod_{i \in I} R_i) = \prod_{i \in I} \Delta(R_i)\) and \(U(\prod_{i \in I} R_i) = \prod_{i \in I} U(R_i)\), the result follows.
\end{proof}

\begin{proposition}\label{3.2}
Let \(R\) be a 2-$\Delta$U ring. If \(T\) is a factor ring of \(R\) such that units of \(T\) lift to units of \(R\), then \(T\) is 2-$\Delta$U.
\end{proposition}
\begin{proof}
Suppose that \(f: R \rightarrow T\) is a ring epimorphism. Let \(v \in U(T)\). Then there exists \(u \in U(R)\) such that \(v = f(u)\) and \(u^2 = 1 + r \in 1 + \Delta(R)\). Thus, we have, \(
v^2 = (f(u))^2 = f(u^2) = f(1 + r) = f(1) + f(r) = 1 + f(r) \in 1 + \Delta(T)\)
\end{proof}

\begin{example}\label{3.3}
A division ring \(R\) is 2-$\Delta$U if, and only if, \(R \cong \mathbb{Z}_2\) or \(\mathbb{Z}_3\).
\end{example}
\begin{proof}
Since \(R\) is a division ring, \(\Delta(R) = 0\), the result follows from \cite [Example 2.1]{1}.
\end{proof}

\begin{remark}
The condition "units of $T$ lift to units of $R$" in Proposition \ref{3.2} is necessary. The ring  $\mathbb{Z}_7$ is a factor ring of the 2-$\Delta U$ ring $\mathbb{Z}$. But, $\mathbb{Z}_7$ is not 2-$\Delta U$ by Example \ref{3.3}. Not that, not all of units of $\mathbb{Z}_7$ can lift to units of $\mathbb{Z}$. 
\end{remark}

\begin{proposition}\label{3.4}
Let \(R\) be a 2-$\Delta$U ring. For a unital subring \(S\) of \(R\), if \(S \cap \Delta(R) \subseteq \Delta(S)\), then \(S\) is a 2-$\Delta$U ring. In particular, the center of \(R\) is a 2-$\Delta$U ring.
\end{proposition}
\begin{proof}
Let \(v \in U(S)\) \(\subseteq U(R)\). Since \(R\) is 2-$\Delta$U, we have,
\(
v^2-1 \in \Delta(R) \cap S \subseteq \Delta(S)
\). So \(S\) is a 2-$\Delta$U ring. The rest follows from \cite [Corollary 8]{2}.
\end{proof}

\begin{proposition}\label{3.27}
Let \( R \) be a 2-$\Delta$U ring and \( 2 \in \Delta(R) \). Then:
\begin{enumerate}
\item 
\( (U(R))^2 + (U(R))^2 \subseteq \Delta(R) \).
\item 
\( [(U(R))^2 + (U(R))^2] \cap \operatorname{Id}(R) = 0 \).
\end{enumerate}
\end{proposition}
\begin{proof}
\begin{enumerate}
\item 
Let $t\in (U(R))^2 + (U(R))^2$, so $t=u^2+v^2$ where $u,v \in U(R)$. Since $R$ is 2-$\Delta$U, $t=1+r+1+s$ where $r,s\in \Delta(R)$. So, we have $t=2+(r+s)$. Since, $2 \in \Delta(R)$ and $\Delta(R)$ is a subring of $R$, $t\in \Delta(R)$.
\item 
It is clear by (i) and \cite [Proposition 15]{2}.
\end{enumerate}
\end{proof}

\begin{theorem}\label{3.5}
Let \(I \subseteq J(R)\) be an ideal of a ring \(R\). Then \(R\) is 2-$\Delta$U if, and only if,  is so \(R/I\).
\end{theorem}
\begin{proof}
Let \(R\) be a 2-$\Delta$U ring and \(u + I \in U(R/I)\). Then \(u \in U(R)\) and hence \(u^2 = 1 + r\) where \(r \in \Delta(R)\). Thus \((u + I)^2 = u^2 + I = (1+I)+(r+I)\), where \(r + I \in \Delta(R)/I = \Delta(R/I)\) by \cite [Proposition 6]{2}. Conversely, let \(R/I\) be a 2-$\Delta$U ring and \(u \in U(R)\). Then \(u + I \in U(R/I)\) and hence \((u + I)^2 = (1 + I) + (r + I)\), where \(r + I \in \Delta(R/I)\). Thus \(u^2 + I = (1 + r) + I\). So \(u^2 - (1 + r) \in I \subseteq J(R) \subseteq \Delta(R)\). Therefore \(u^2 = 1 + r^\prime\), where \(r^\prime \in \Delta(R)\). Hence \(R\) is a 2-$\Delta$U ring.
\end{proof}

\begin{corollary}\label{3.6}
A ring \(R\) is 2-$\Delta$U if, and only if, \(R/J(R)\) is 2-$\Delta$U.
\end{corollary}

\begin{proposition}\label{3.7}
Let \(R\) be a 2-$\Delta$U ring and \(e\) be an idempotent of \(R\). Then \(eRe\) is a 2-$\Delta$U ring.
\end{proposition}
\begin{proof}
Let \(u \in U(eRe)\). Then \(u + (1-e) \in U(R)\). By the hypothesis, \((u + (1-e))^2 = u^2 + (1 - e) = 1 + r \in 1 + \Delta(R)\). Then we have \(u^2 - e \in \Delta(R)\). Now we show that \(u^2 - e \in \Delta(eRe)\). Let \(v\) be an arbitrary unit of \(eRe\). Clearly, \(v + 1 - e \in U(R)\). Note that \(u^2 - e \in \Delta(R)\) gives us that \(u^2 - e + v + 1 - e \in U(R)\), by the definition of \(\Delta(R)\). Take \(u^2 - e + v + 1 - e = t \in U(R)\).
One can check that \( e t = t e = ete = u^2 - e + v \), and so \( e t e \in U(eRe) \). It shows that \( u^2 - e + U(eRe) \subseteq U(eRe) \). Then we have \( u^2 - e \in \Delta(eRe) \) and hence \( u^2\in e + \Delta(eRe) \), which implies \( eRe \) is a 2-$\Delta$U ring.
\end{proof}

\begin{proposition}\label{3.8}
For any ring \( R \neq 0 \) and any integer \( n \geq 2 \), the ring \( M_n(R) \) is not a 2-$\Delta$U ring.
\end{proposition}
\begin{proof}
Since it is well known that \( M_2(R) \) is isomorphic to a corner ring of \( M_n(R) \) (for \( n \geq 2 \)), it suffices to show that \( M_2(R) \) is not a 2-$\Delta$U ring by virtue of Proposition \ref{3.7}. To this target, consider the matrix 
\[
A = \begin{pmatrix} 0 & 1 \\ 1 & 1 \end{pmatrix} \in U(M_2(R)).
\]
Then, \(A^2 - I_2 = A \not \in J(M_2(R)) = \Delta(M_2(R))\), as required.
\end{proof}

A set $\{e_{ij} : 1 \le i, j \le n\}$ of nonzero elements of $R$ is said to be a system of $n^2$ matrix units if $e_{ij}e_{st} = \delta_{js}e_{it}$, where $\delta_{jj} = 1$ and $\delta_{js} = 0$ for $j \neq s$. In this case, $e := \sum_{i=1}^{n} e_{ii}$ is an idempotent of $R$ and $eRe \cong M_n(S)$, where $$S = \{r \in eRe : re_{ij} = e_{ij}r,~~\textrm{for all}~~ i, j = 1, 2, . . . , n\}.$$
Recall that, a ring $R$ is said to be a Dedekind-finite if $ab=1$ implies $ba=1$ for any $a,b\in R$. In other words, all one-sided inverses in the ring are two-sided.

\begin{proposition}\label{3.28}
Every 2-$\Delta$U ring is Dedekind-finite.
\end{proposition}
\begin{proof}
If we assume to the contrary that $R$ is not a Dedekind-finite ring, then there exist elements $a, b \in R$ such that $ab = 1$ but $ba \neq 1$. Assuming $e_{ij} = a^i(1-ba)b^j$ and $e =\sum_{i=1}^{n}e_{ii}$, there exists a nonzero ring $S$ such that $eRe \cong M_n(S)$. However, owing to Proposition \ref{3.7}, $eRe$ is a 2-$\Delta$U ring, so $M_n(S)$ must also be a 2-$\Delta$U ring, which contradicts Proposition \ref{3.8}.
\end{proof}

\begin{example}\label{3.10}
A local ring \( R \) is 2-$\Delta$U if, and only if, \( R/J(R) \cong \mathbb{Z}_2 \) or \( \mathbb{Z}_3 \).
\end{example}
\begin{proof}
Assume that \( R/J(R) \cong \mathbb{Z}_2 \) or \( \mathbb{Z}_3 \). We know that \(R/J(R)\) is a division ring, so \(R/J(R)\) is 2-$\Delta$U by Example \ref{3.3}. Thus, \(R\) is 2-$\Delta$U by Corollary \ref{3.6}. Conversely, let \(R\) be 2-$\Delta$U and hence \(R/J(R) \cong \mathbb{Z}_2\) or \(\mathbb{Z}_3\) by Example \ref{3.3}.
\end{proof}

\begin{corollary}\label{3.26}
(i)A semi-simple ring $R$ is 2-$\Delta$U if, and only if, $R \cong \bigoplus_{i=1}^n R_i$ where $R_i\cong \mathbb{Z}_2 \text{ or } \mathbb{Z}_3$ for every i.

(ii)A semi-local ring $R$ is 2-$\Delta$U if, and only if, $R/J(R) \cong \bigoplus_{i=1}^m R_i$ where $R_i \cong \mathbb{Z}_2 \text{ or } \mathbb{Z}_3$ for every i.
\end{corollary}

\begin{example}\label{3.11}
The ring \(\mathbb{Z}_m\) is 2-$\Delta$U if, and only if, \(m = 2^k 3^l\) for some positive integers \(k\) and \(l\).
\end{example}

\begin{lemma}\label{3.12}
Let \(R\) be a 2-$\Delta$U ring. If \(J(R) = 0\) and every nonzero right ideal of \(R\) contains a nonzero idempotent, then \(R\) is reduced.
\end{lemma}
\begin{proof}
Suppose that \(R\) is not reduced. Then there exists a nonzero element \(a \in R\) such that \(a^2 = 0\). By \cite [Theorem 2.1]{3}, there is an idempotent \(e \in RaR\) such that \(eRe \cong M_2(T)\) for some non-trivial ring \(T\). By Proposition \ref{3.7}, $eRe$ is a 2-$\Delta$U ring and hence $M_2(T)$ is a 2-$\Delta$U ring. This contradicts to Proposition \ref{3.8}.
\end{proof}

A ring $R$ is called $\pi$-regular if for each $a\in R$, $a^n\in a^nRa^n$ for some integer $n\ge1$. Regular rings are $\pi$-regular. Also, a ring $R$ is said to be strongly $\pi$-regular provided that for any $a\in R$ there exists $n\ge1$, such that $a^n\in a^{n+1}R$.

\begin{theorem}\label{3.13}
Let \(R\) be a ring. The following are equivalent:
\begin{enumerate}
\item 
\(R\) is a regular 2-$\Delta$U ring.
\item 
\(R\) is a \(\pi\)-regular reduced 2-$\Delta$U ring.
\item 
\(R\) has the identity \(x^3 = x\) (i.e., \(R\) is a tripotent ring).
\end{enumerate}
\end{theorem}
\begin{proof}
\((i) \Rightarrow (ii)\) Since \(R\) is regular, \(J(R) = 0\), and every nonzero right ideal contains a nonzero idempotent. In view of Lemma \ref{3.12}, \(R\) is reduced. Also, clearly every regular ring is \(\pi\)-regular.

\((ii) \Rightarrow (iii)\) Notice that reduced rings are abelian, so \(R\) is abelian regular by \cite [Theorem 3]{4} and hence it is strongly regular. Then \(R\) is unit-regular, so \(\Delta(R) = 0\) by \cite [Corollary 16]{2}. Thus, we have \(Nil(R) = J(R) = \Delta(R) = 0\). On the other hand, \(R\) is strongly \(\pi\)-regular. Let \(x \in R\). In view of \cite [Proposition 2.5]{5}, there is an idempotent \(e \in R\) and a unit \(u \in R\) such that \(x = e + u\), \(ex = xe \in Nil(R) = 0\). So we have \(x = x - xe = x(1-e) = u(1-e) = (1-e)u\). Since \(R\) is a 2-$\Delta$U ring, \(u^2 = 1\). It follows that \(x^2 = (1-e)\). Hence, \(x = x(1-e) = x.x^2 = x^3\).

\((iii) \Rightarrow (i)\) It is clear that \(R\) is regular. Let \(u \in U(R)\). Then we have \(u^3 = u\), which implies that \(u^2 = 1\), and thus, \(R\) is a 2-$\Delta$U ring.
\end{proof}

A ring $R$ is strongly $2$-nil-clean if every element in $R$ is the sum of two idempotents and a nilpotent that commute (for details see \cite{6}).

\begin{theorem}\label{3.14}
The following are equivalent for a ring \(R\):
\begin{enumerate}
\item 
\(R\) is a regular 2-$\Delta$U ring.
\item
 \(R\) is a strongly regular 2-$\Delta$U ring.
\item 
\(R\) is a unit-regular 2-$\Delta$U ring.
\item 
\(R\) has the identity \(x^3 = x\).
\end{enumerate}
\end{theorem}
\begin{proof}
\((i) \Rightarrow (ii)\) In view of Lemma \ref{3.12}, \(R\) is reduced and hence it is abelian. Then \(R\) is strongly regular.

\((ii) \Rightarrow (iii)\) This is obvious.

\((iii) \Rightarrow (iv)\) Let \(x \in R\). Then, \(x = u e\) for some \(u \in U(R)\) and \(e \in Id(R)\). We know that every unit-regular ring is regular, so \(R\) is regular 2-$\Delta$U and hence \(R\) is abelian. On the other hand, by \cite [Corollary 16]{2}, we have \(\Delta(R) = 0\). Therefore, for any \(u \in U(R)\), we have \(u^2 = 1\). Then, \(x^2 = u^2e^2 = e\). So, \(R\) is a 2-Boolean ring. Thus, by \cite [Corollary 3.4]{6}, \(R\) is strongly \(2\)-nil-clean and hence by \cite [Theorem 3.3]{6}, \(R\) is tripotent.

\((iv) \Rightarrow (i)\) It is clear by Theorem \ref{3.13}.
\end{proof}

\begin{proposition}\label{3.15}
A ring \(R\) is a $\Delta$U ring if, and only if,
\begin{enumerate}
\item 
\(2 \in \Delta(R)\),
\item 
\(R\) is a 2-$\Delta$U ring,
\item
 If \(x^2 \in \Delta(R)\), then \(x \in \Delta(R)\) for every \(x \in R\).
\end{enumerate}
\end{proposition}
\begin{proof}
\(\Rightarrow\) As \(R\) is a $\Delta$U ring, then \(-1 = 1 + r\) for some \(r \in \Delta(R)\). This implies that \(-2 \in \Delta(R)\) and so \(2 \in \Delta(R)\). Clearly every $\Delta$U ring is 2-$\Delta$U. The result follows from \cite [Proposition 2.4]{7}.

\(\Leftarrow\) Let \(u \in U(R)\). Then \((u-1)^2 + 2(u-1) = (u-1)(u+1) = u^2-1 \in \Delta(R)\) as \(R\) is a 2-$\Delta$U ring. It follows from \(2 \in \Delta(R)\) and \(\Delta(R)\) is a subring of \(R\) that \((u-1)^2 \in \Delta(R)\).
So by (iii), we have \(u - 1 \in \Delta(R)\) and hence \(R\) is a $\Delta$U ring.
\end{proof}

Due to Kosan et al. \cite {17}, a ring $R$ is called semi-tripotent if for each $a\in R$, $a=e+j$ where $e^3=e$ and $j\in J(R)$ (equivalently, $R/J(R)$ has the identity $x^3=x$ and idempotents lift modulo $J(R)$).

\begin{theorem}\label{3.16}
Let \(R\) be a ring. Then, the following are equivalent:
\begin{enumerate}
\item 
\(R\) is a semi-regular 2-$\Delta$U ring.
\item 
\(R\) is an exchange 2-$\Delta$U ring.
\item 
\(R\) is a semi-tripotent ring.
\end{enumerate}
\end{theorem}
\begin{proof} 
\((i) \Rightarrow (ii)\) In view of \cite [Proposition 1.6]{8}, semi-regular rings are exchange.

\((ii) \Rightarrow (iii)\) By \cite [Corollary 2.4]{8}, \(R/J(R)\) is exchange and idempotents lift modulo \(J(R)\). Also, by Theorem \ref {3.5}, \(R/J(R)\) is 2-$\Delta$U. So, without loss of generality, it can be assumed that \(J(R) = 0\). Since \(R\) is an exchange ring, every nonzero one sided ideal contains a nonzero idempotent. By Lemma \ref{3.12}, \(R\) is reduced and hence it is abelian. Thus, \(R\) is abelian weakly clean. Hence by \cite [Proposition 14]{9}, \(R/J(R) \cong M_n(D)\) where \(1 \leq n \leq 2\) and \(D\) is a division ring. Then by \cite [Theorem 11]{2}, we have \(\Delta(R) = J(R)\) and hence \(\Delta(R) = 0\). As \(R\) is 2-$\Delta$U, we have \(v^2 = 1\) for every \(v \in U(R)\). Then the result follows from \cite [Theorem 3.3]{1}.

\((iii) \Rightarrow (i)\) By \cite [Theorem 3.3]{1}, \(R\) is semi-regular 2-$UJ$ and hence \(R\) is semi-regular 2-$\Delta$U.
\end{proof}

\begin{corollary}\label{3.17}
Let \( R \) be a 2-$\Delta$U ring. Then the following are equivalent:
\begin{enumerate}
\item 
\( R \) is a semi-regular ring.
\item 
\( R \) is an exchange ring.
\item 
\( R \) is a clean ring.
\end{enumerate}
\end{corollary}
\begin{proof}
By Theorem \ref{3.16}, (i) \(\Leftrightarrow\) (ii). 

\((iii) \Rightarrow (ii)\) This is clear.

\((ii) \Rightarrow (iii)\) If \( R \) is exchange 2-$\Delta$U, then \( R \) is reduced by Lemma \ref{3.12} and hence it is abelian. Therefore, \( R \) is abelian exchange, so it is clean.
\end{proof}

\begin{corollary}\label{3.18}
Let \( R \) be a ring. The following are equivalent:
\begin{enumerate}
\item 
\( R \) is a semi-regular 2-$\Delta$U ring and \(J(R)\) is nil.
\item 
\( R \) is an exchange 2-$\Delta$U ring and \(J(R)\) is nil.
\item 
\( R \) is a strongly $2$-nil-clean ring.
\end{enumerate}
\end{corollary}
\begin{proof} 
\((i) \Rightarrow (ii)\) This is clear.

\((ii) \Rightarrow (iii)\) Since \( R \) is exchange 2-$\Delta$U, \( \Delta(R) = J(R) \). Then, for any \( u \in U(R) \), we have \( u^2 - 1 \in \Delta(R) = J(R) \subseteq Nil(R)\). So, \( R \) is 2-$UU$ ring. Therefore, \( R \) is exchange 2-$UU$ ring and hence it is strongly \( 2 \)-nil-clean by \cite [Theorem 4.1]{10}.

\((iii) \Rightarrow (i)\) This is clear by \cite [Corollary 3.5]{1} and we know that every $2$-UJ ring is $2$-$\Delta$U.
\end{proof}

\begin{corollary}\label{3.19}
Let \(R\) be a ring. Then, the following are equivalent:
\begin{enumerate}
\item 
\(R\) is a regular $\Delta$U ring.
\item 
\(R\) is a \(\pi\)-regular reduced $\Delta$U ring.
\item 
 \(R\) is a Boolean ring.
\end{enumerate}
\end{corollary}
\begin{proof}
This is obvious by Theorem \ref{3.13} and \cite [Theorem 4.4]{7}.	
\end{proof}	

\begin{corollary}\label{3.20}
Let \(R\) be a ring. The following are equivalent:
\begin{enumerate}
\item 
\(R\) is a semi-regular \(\Delta U\) ring.
\item 
\(R\) is an exchange $\Delta$U ring.
\item 
\(R\) is a \(\Delta\)-clean ring.
\end{enumerate}
\end{corollary}
\begin{proof}
The result follows from \cite [Theorem 4.2 and Corollary 4.7]{7}.
\end{proof}

\section{Some extensions of 2-$\Delta U$ rings}

We say that $C$ is an unital subring of a ring $D$ if $\emptyset\neq C\subseteq D$ and, for any $x,y\in C$, the relations $x-y$, $xy\in C$ and $1_{D}\in C$ hold. Let $D$ be a ring and $C$ an unital subring of $D$, and denote by $R[D,C]$ the set $\lbrace (d_{1},\ldots ,d_{n},c,c,\ldots ): d_{i}\in D,c\in C,1\leq i\leq n \rbrace$. Then, $R[D,C]$ forms a ring under the usual component-wise addition and multiplication. The ring $R[D,C]$ is called the tail ring extension.

\begin{proposition}\label{3.25}
\(R[D, C]\) is a 2-$\Delta$U ring if, and only if, \(D\) and \(C\) are 2-$\Delta$U rings.
\end{proposition}
\begin{proof}
Let $R[D, C]$ be a 2-$\Delta$U ring. Firstly, we prove that $D$ is a 2-$\Delta$U ring. Let $u\in U(D)$. Then $\bar{u} = (u,1, 1, \ldots)\in U(R[D, C])$. By the hypothesis, we have $(u^2 - 1, 0, 0, \ldots) \in \Delta(R[D, C])$, so $(u^2-1, 0, 0, \ldots) + U(R[D, C]) \subseteq U(R[D, C])$. Thus, for all $v \in U(D)$, $(u^2 -1 + v,1, 1, \ldots) = (u^2 - 1,0, 0, \ldots) + (v, 1, 1, \ldots) \in U(R[D, C])$.
Hence, $u^2 -1 + v \in U(D)$, which implies that $u^2 -1 \in \Delta(D)$. Now, we show that $C$ is a 2-$\Delta$U ring. Let $v \in U(C)$. Then $(1, \ldots, 1, 1, v, v, \dots)\in U(R[D, C])$. By hypothesis, $(0, \ldots, 0, v^2 - 1, v^2 - 1, \dots) \in \Delta(R[D, C])$, so $(0, \ldots, 0, v^2 - 1, v^2 -1, \dots) + U(R[D, C]) \subseteq U(R[D, C])$. Thus, for all $u \in U(C)$, $(1,1, \dots, v^2 -1 + u, v^2 -1 + u, \dots) \in U(R[D, C])$. Then, we have $v^2 - 1 + u \in U(C)$ and hence $v^2 -1 \in \Delta(C)$. For the converse, assume that $D$ and $C$ are 2-$\Delta$U rings. Let $\bar{u} = (u_1, u_2, \ldots, u_n, v, v, \ldots) \in U(R[D, C])$, where $u_i \in U(D)$ and $v \in U(C) \subseteq U(D)$. We must show that $\bar{u}^{2} - 1 + U(R[D, C]) \subseteq U(R[D, C])$. In fact, for all $\bar{a} = (a_1, \ldots, a_m, b, b, \ldots) \in U(R[D, C])$ with $a_i \in U(D)$ and $b \in U(C) \subseteq U(D)$, take $z = \max \{m, n\}$, then we have, $\bar{u}^2 - 1 + \bar{a} = (u_1^2-1 + a_1, \dots, u_z^2 - 1 + a_z, v^2 - 1 + b, v^2 - 1 + b, \dots)$. Note that $u_i^2 - 1 + a_i \in U(D)$ for all $1 \leq i \leq z$ and $v^2 - 1 + b \in U(C) \subseteq U(D)$. We deduce that $\bar{u}^2 - 1 + \bar{a} \in U(R[D, C])$. Thus, $\bar{u}^2 - 1\in \Delta(R[D, C])$ or $\bar{u}^2\in 1 + \Delta(R[D, C])$. This shows that $R[D, C]$ is a 2-$\Delta$U ring.	
\end{proof}

Let $R$ be a ring and $\alpha : R \to R$ is a ring endomorphism and $R[[x; \alpha]]$ denotes the ring of skew formal power series over $R$; that is all formal power series in $x$ with coefficients from $R$ with multiplication defined by $xr = \alpha(r)x$ for all $r \in R$. In particular, $R[[x]] = R[[x; 1_R]]$ is the ring of formal power series over $R$.

\begin{proposition}\label{3.9}
A ring $R$ is 2-$\Delta$U if, and only if, so is $R[[x; \alpha]]$.
\end{proposition}
\begin{proof}
Consider $I= R[[x; \alpha]]x$. Then $I$ is an ideal of $R[[x; \alpha]]$. Note that $J(R[[x; \alpha]])=J(R)+I$, so $I\subseteq J(R[[x; \alpha]])$. Since $R[[x; \alpha]]/I\cong R$, the result follows by Theorem \ref{3.5}.
\end{proof}

\begin{corollary}
A ring $R$ is 2-$\Delta$U if, and only if, so is $R[[x]]$.
\end{corollary}

\begin{theorem}\label{3.29}
Let \( R \) be a 2-primal ring and \( \alpha \) be an endomorphism of \( R \) such that \( R \) is \( \alpha \)-compatible. The following are equivalent:
\begin{enumerate}
\item 
\( R[x; \alpha] \) is a 2-$\Delta$U ring.
\item 
\( R \) is a 2-$\Delta$U ring.
\end{enumerate}
\end{theorem}
\begin{proof}
\((ii) \Rightarrow (i)\) Let \( u(x) = a_0 + a_1 x + \dots + a_n x^n = \sum_{i=0}^n a_i x^i \) in \( U(R[x; \alpha]) \). So by \cite [Corollary 2.14]{15}, \( a_0 \in U(R) \) and \( a_i \in \operatorname{Nil}(R) \) for each \( i \geq 1 \). Then by hypothesis \( a_0^2 = 1 + r \) where \( r \in \Delta(R) \). On the other hand, we know that \(J(R[x; \alpha]) = \operatorname{Nil}_*(R[x; \alpha]) = \operatorname{Nil}_*(R)[x; \alpha] = \operatorname{Nil}(R)[x; \alpha]\).
Now, we have
	\[
	(u(x))^2 = a_0^2 + a_0 a_1 x + \dots + a_0 a_n x^n + a_1 x a_0 + \dots = (1 + r) + a_0 a_1 x + \dots
	\]
\[
	  = 1 + (r + a_0 a_1 x + \dots) \in 1 + \Delta(R) + J(R[x; \alpha]).
\]
On the other hand, we have $\Delta(R)+J(R[x; \alpha])=\Delta(R[x; \alpha])$ by Proposition \ref{2.10}. Then, this shows that \( R[x; \alpha] \) is a 2-\(\Delta U\) ring.\\
\((i) \Rightarrow (ii)\) Let \( u \in U(R) \subseteq U(R[x; \alpha]) \). Then \( u^2 \in 1 + \Delta(R[x; \alpha]) = 1 + \Delta(R) + J(R[x; \alpha]) \). Thus, we have \( u^2 \in 1 + \Delta(R) \) and hence \( R \) is a 2-$\Delta$U ring.
\end{proof}

\begin{corollary}\label{3.30}
Let \( R \) be a \(2\)-primal ring. Then, the following are equivalent:
\begin{enumerate}
\item 
\( R[x] \) is a 2-$\Delta$U ring.
\item
\( R \) is a 2-$\Delta$U ring.
\end{enumerate}
\end{corollary}

Let $R$ be a ring and $M$ a bi-modulo over $R$. The trivial extension of $R$ and $M$ is defined as
\[ T(R, M) = \{(r, m) : r \in R \text{ and } m \in M\}, \]
with addition defined componentwise and multiplication defined by
\[ (r, m)(s, n) = (rs, rn + ms). \]
The trivial extension $T(R, M)$ is isomorphic to the subring
\[ \left\{ \begin{pmatrix} r & m \\ 0 & r \end{pmatrix} : r \in R \text{ and } m \in M \right\} \]
of the formal $2 \times 2$ matrix ring $\begin{pmatrix} R & M \\ 0 & R \end{pmatrix}$, and also $T(R, R) \cong R[x]/\left\langle x^2 \right\rangle$. We also note that the set of units of the trivial extension $T(R, M)$ is
\[ U(T(R, M)) = T(U(R), M). \]
Also by \cite {7}, we have 
\[ \Delta(T(R, M)) = T(\Delta(R), M). \]

\begin{proposition}\label{4.5}
Let $R$ be a ring and $M$ a bi-modulo over $R$. Then, the following hold:
\begin{enumerate}
\item
The trivial extension ${\rm T}(R, M)$ is a 2-$\Delta U$ ring if, and only if, $R$ is a 2-$\Delta U$ ring.
\item
For $n \geq 2$, the quotient-ring $\dfrac{R[x; \alpha]}{\langle x^n\rangle}$ is a 2-$\Delta U$ ring if, and only if, $R$ is a 2-$\Delta U$ ring.
\item
For $n \geq 2$, the quotient-ring $\dfrac{R[[x; \alpha]]}{\langle x^n\rangle}$ is a 2-$\Delta U$ ring if, and only if, $R$ is a 2-$\Delta U$ ring.
\item
The upper triangular matrix ring $T_n(R)$ is a 2-$\Delta U$ if, and only if, $R$ is a 2-$\Delta U$ ring.
\end{enumerate}
\end{proposition}
\begin{proof}
\begin{enumerate}
\item
Set $A={\rm T}(R, M)$ and consider $I:={\rm T}(0, M)$. It is not too hard to verify that $I\subseteq J(A)$ such that $\dfrac{A}{I} \cong R$. So, the result follows directly from Theorem \ref{3.5}.
\item
Put $A=\dfrac{R[x; \alpha]}{\langle x^n\rangle}$. Considering $I:=\dfrac{\langle x\rangle}{\langle x^n\rangle}$, we obtain that $I\subseteq J(A)$ such that $\dfrac{A}{I} \cong R$. So, the result follows automatically Theorem \ref{3.5}.
\item
Knowing that the isomorphism $\dfrac{R[x; \alpha]}{\langle x^n\rangle} \cong \dfrac{R[[x; \alpha]]}{\langle x^n\rangle}$ is true,  (iii) follows automatically from (ii).
\item
Let $I=\left\lbrace (a_{ij})\in T_{n}(R)\, | \, a_{ii}=0\right\rbrace$. Then, we have $I\subseteq J(T_{n}(R))$ with $T_{n}(R)/I\cong R^n$. Therefore, the desired result follows from Theorem \ref{3.5} and Proposition \ref{3.1}.
\end{enumerate}
\end{proof}

\begin{corollary}\label{4.6}
Let $R$ be a ring. Then, the following are equivalent:
\begin{enumerate}
\item 
$R$ is a 2-$\Delta$U ring.
\item
For $n \geq 2$, the quotient-ring $\dfrac{R[x]}{\langle x^n\rangle}$ is a 2-$\Delta U$ ring.
\item
For $n \geq 2$, the quotient-ring $\dfrac{R[[x]]}{\langle x^n\rangle}$ is a 2-$\Delta U$ ring.
\end{enumerate}
\end{corollary}

\begin{example}
The upper triangular ring $T_{n}(\mathbb{Z}_3)$ for all $n\geq 1$ is 2-$\Delta U$ by Proposition \ref{4.5}(4) and Example \ref{3.32}. But, it is not $\Delta U$ by Example \ref{3.32} and \cite [Corollary 2.9]{7}.
\end{example}

Let $R$  be a ring and $M$  a  bi-modulo over $R$. Let $${\rm DT}(R,M) := \{ (a, m, b, n) | a, b \in R, m, n \in M \}$$ with addition defined componentwise and multiplication defined by $$(a_1, m_1, b_1, n_1)(a_2, m_2, b_2, n_2) = (a_1a_2, a_1m_2 + m_1a_2, a_1b_2 + b_1a_2, a_1n_2 + m_1b_2 + b_1m_2 +n_1a_2).$$ Then, ${\rm DT}(R,M)$ is a ring which is isomorphic to ${\rm T}({\rm T}(R, M), {\rm T}(R, M))$. Also, we have $${\rm DT}(R, M) =
\left\{\begin{pmatrix}
	a &m &b &n\\
	0 &a &0 &b\\
	0 &0 &a &m\\
	0 &0 &0 &a
\end{pmatrix} |  a,b \in R, m,n \in M\right\}.$$ We have the following isomorphism as rings: $\dfrac{R[x, y]}{\langle x^2, y^2\rangle} \rightarrow {\rm DT}(R, R)$ defined by $$a + bx + cy + dxy \mapsto
\begin{pmatrix}
	a &b &c &d\\
	0 &a &0 &c\\
	0 &0 &a &b\\
	0 &0 &0 &a
\end{pmatrix}.$$

We, thereby, detect the following.

\begin{corollary}
Let $R$ be a ring and $M$ a bi-modulo over $R$. Then the following statements are equivalent:
\begin{enumerate}
\item
$R$ is a 2-$\Delta U$ ring.
\item
${\rm DT}(R, M)$ is a 2-$\Delta U$ ring.
\item
${\rm DT}(R, R)$ is a 2-$\Delta U$ ring.
\item
$\dfrac{R[x, y]}{\langle x^2, y^2\rangle}$ is a 2-$\Delta U$ ring.
\end{enumerate}
\end{corollary}

Let $A$, $B$ be two rings and $M$, $N$ be $(A,B)$-bi-modulo and $(B,A)$-bi-modulo, respectively. Also, we consider the bi-linear maps $\phi :M\otimes_{B}N\rightarrow A$ and $\psi:N\otimes_{A}M\rightarrow B$ that apply to the following properties.
$$Id_{M}\otimes_{B}\psi =\phi \otimes_{A}Id_{M},Id_{N}\otimes_{A}\phi =\psi \otimes_{B}Id_{N}.$$
For $m\in M$ and $n\in N$, define $mn:=\phi (m\otimes n)$ and $nm:=\psi (n\otimes m)$. Now the $4$-tuple $R=\begin{pmatrix}
	A & M\\
	N & B
\end{pmatrix}$ becomes to an associative ring with obvious matrix operations that is called a {\it Morita context ring}. Denote two-side ideals $Im \phi$ and $Im \psi$ to $MN$ and $NM$, respectively, that are called the {\it trace ideals} of the {\it Morita context ring}.

\begin{proposition}\label{4.7}
Let $R=\left(\begin{array}{ll}A & M \\ N & B\end{array}\right)$ be a Morita context ring such that $MN \subseteq J(A)$ and $NM\subseteq J(B)$. Then, $R$ is a 2-$\Delta U$ ring if, and only if, both $A$ and $B$ are 2-$\Delta U$.
\end{proposition}
\begin{proof}
In view of \cite [Lemma 3.1]{26}, we argue that $J(R)=\begin{pmatrix}
	J(A) & M \\
	N & J(B)
\end{pmatrix}$ and hence $\dfrac{R}{J(R)}\cong \dfrac{A}{J(A)}\times \dfrac{B}{J(B)}$. Then, the result follows from Corollary \ref{3.6} and Proposition \ref{3.1}.
\end{proof}

Now, let $R$, $S$ be two rings, and let $M$ be an $(R,S)$-bi-modulo such that the operation $(rm)s = r(ms$) is valid for all $r \in R$, $m \in M$ and $s \in S$. Given such a bi-modulo $M$, we can set

$$
{\rm T}(R, S, M) =
\begin{pmatrix}
	R& M \\
	0& S
\end{pmatrix}
=
\left\{
\begin{pmatrix}
	r& m \\
	0& s
\end{pmatrix}
: r \in R, m \in M, s \in S
\right\},
$$
where it forms a ring with the usual matrix operations. The so-stated formal matrix ${\rm T}(R, S, M)$ is called a {\it formal triangular matrix ring}. In Proposition \ref{4.7}, if we set $N=0$, then we will obtain the following corollary.

\begin{corollary}\label{4.8}
Let $R,S$ be rings and let $M$ be an $(R,S)$-bi-modulo. Then, the formal triangular matrix ring ${\rm T}(R,S,M)$ is a 2-$\Delta U$ ring if, and only if, both $R$ and $S$ are 2-$\Delta U$.
\end{corollary}

Given a ring $R$ and a central element $s$ of $R$, the $4$-tuple $\begin{pmatrix}
	R & R\\
	R & R
\end{pmatrix}$ becomes a ring with addition component-wise and with multiplication defined by
$$\begin{pmatrix}
	a_{1} & x_{1}\\
	y_{1} & b_{1}
\end{pmatrix}\begin{pmatrix}
	a_{2} & x_{2}\\
	y_{2} & b_{2}
\end{pmatrix}=\begin{pmatrix}
	a_{1}a_{2}+sx_{1}y_{2} & a_{1}x_{2}+x_{1}b_{2} \\
	y_{1}a_{2}+b_{1}y_{2} & sy_{1}x_{2}+b_{1}b_{2}
\end{pmatrix}.$$
This ring is denoted by $K_s(R)$. A {\it Morita context}
$\begin{pmatrix}
	A & M\\
	N & B
\end{pmatrix}$ with $A=B=M=N=R$ is called a {\it generalized matrix ring} over $R$. It was observed by Krylov in \cite{18} that a ring $S$ is a generalized matrix ring over $R$ if, and only if, $S=K_s(R)$ for some $s\in C(R)$. Here $MN=NM=sR$, so $MN\subseteq J(A)\Longleftrightarrow s\in J(R)$, $NM\subseteq J(B)\Longleftrightarrow s\in J(R)$.

\begin{corollary}\label{4.9}
Let $R$ be a ring and $s\in C(R)\cap J(R)$. Then, $K_s(R)$ is a 2-$\Delta U$ ring if, and only if, $R$ is 2-$\Delta U$.
\end{corollary}

Following Tang and Zhou (cf. \cite{19}), for $n\geq 2$ and for $s\in C(R)$, the $n\times n$ formal matrix ring over $R$ defined by $s$, and denoted by $M_{n}(R;s)$, is the set of all $n\times n$ matrices over $R$ with usual addition of matrices and with multiplication defined below:

\noindent For $(a_{ij})$ and $(b_{ij})$ in ${\rm M}_{n}(R;s)$,
$$(a_{ij})(b_{ij})=(c_{ij}), \quad \text{where} ~~ (c_{ij})=\sum s^{\delta_{ikj}}a_{ik}b_{kj}.$$
Here, $\delta_{ijk}=1+\delta_{ik}-\delta_{ij}-\delta_{jk}$, where $\delta_{jk}$, $\delta_{ij}$, $\delta_{ik}$ are the Kroncker delta symbols.

\begin{corollary}\label{4.10}
Let $R$ be a ring and $s\in C(R)\cap J(R)$. Then, $M_{n}(R;s)$ is a 2-$\Delta U$ ring if, and only if, $R$ is 2-$\Delta U$.
\end{corollary}
\begin{proof}
If $n = 1$, then ${\rm M}_n(R;s) = R$. So, in this case, there is nothing to prove. Let $n=2$. By the definition of ${\rm M}_n(R;s)$, we have ${\rm M}_2 (R;s) \cong {\rm K}_{s^2} (R)$. Apparently, $s^2 \in J(R) \cap C(R)$, so the claim holds for $n = 2$ with the help of Corollary \ref{4.9}.
	
To proceed by induction, assume now that $n>2$ and that the claim holds for ${\rm M}_{n-1} (R;s)$. Set $A := {\rm M}_{n-1} (R;s)$. Then, ${\rm M}_n (R;s) =
\begin{pmatrix}
		A & M\\
		N & R
\end{pmatrix}$
is a {\it Morita context}, where $$M =
\begin{pmatrix}
		M_{1n}\\
		\vdots\\
		M_{n-1, n}
\end{pmatrix}
	\quad \text{and} \quad  N = (M_{n1} \dots M_{n, n-1})$$ with $M_{in} = M_{ni} = R$ for all $i = 1, \dots, n-1,$ and
\begin{align*}
		&\psi: N \otimes M \rightarrow N, \quad n \otimes m \mapsto snm\\
		&\phi : M \otimes N \rightarrow M, \quad  m \otimes n \mapsto smn.
\end{align*}
Besides, for $x =
\begin{pmatrix}
		x_{1n}\\
		\vdots\\
		x_{n-1, n}
\end{pmatrix}
	\in M$ and $y = (y_{n1} \dots y_{n, n-1}) \in N$, we write $$xy =
\begin{pmatrix}
		s^2x_{1n}y_{n1} & sx_{1n}y_{n2} & \dots & sx_{1n}y_{n, n-1}\\
		sx_{2n}y_{n1} & s^2x_{2n}y_{n2} & \dots & sx_{2n}y_{n, n-1}\\
		\vdots & \vdots &\ddots & \vdots\\
		sx_{n-1, n}y_{n1} & sx_{n-1, n}y_{n2} & \dots & s^2x_{n-1, n}y_{n, n-1}
\end{pmatrix} \in sA$$ and $$yx = s^2y_{n1}x_{1n} + s^2y_{n2}x_{2n} + \dots + s^2y_{n, n-1}x_{n-1, n} \in s^2 R.$$ Since $s \in J(R)$, we see that $MN \subseteq J(A)$ and $NM\subseteq J(A)$. Thus, we obtain that $$\frac{{\rm M}_n (R; s)}{J({\rm M}_n (R; s))} \cong \frac{A}{J (A)} \times \frac{R}{J (R)}.$$ Finally, the induction hypothesis and Proposition \ref{4.7} yield the claim after all.
\end{proof}

A {\it Morita context} $\begin{pmatrix}
	A & M\\
	N & B
\end{pmatrix}$ is called {\it trivial}, if the context products are trivial, i.e., $MN=0$ and $NM=0$. We now have
$$\begin{pmatrix}
	A & M\\
	N & B
\end{pmatrix}\cong {\rm T}(A\times B, M\oplus N),$$
where
$\begin{pmatrix}
	A & M\\
	N & B
\end{pmatrix}$ is a trivial {\it Morita context} consulting with \cite{20}.

\begin{corollary}\label{4.11}
The trivial Morita context
$\begin{pmatrix}
		A & M\\
		N & B
\end{pmatrix}$ is a 2-$\Delta U$ ring if, and only if, both $A$ and $B$ are 2-$\Delta U$.
\end{corollary}
\begin{proof}
It is apparent to see that the isomorphisms
$$\begin{pmatrix}
		A & M\\
		N & B
\end{pmatrix} \cong {\rm T}(A\times B,M\oplus N) \cong \begin{pmatrix}
		A\times B & M\oplus N\\
		0 & A \times B
\end{pmatrix}$$ are fulfilled. Then, the rest of the proof follows combining Proposition \ref{4.5}(i) and \ref{3.1}.
\end{proof}

As usual, for an arbitrary ring $R$ and an arbitrary group $G$, the symbol $RG$ stands for the {\it group ring} of $G$ over $R$. Standardly, $\varepsilon(RG)$ designates the kernel of the classical {\it augmentation map} $\varepsilon: RG\to R$, defined by $\varepsilon (\displaystyle\sum_{g\in G}a_{g}g)=\displaystyle\sum_{g\in G}a_{g}$, and this ideal is called the {\it augmentation ideal} of $RG$.
A group $G$ is called a {\it $p$-group} if every element of $G$ is a power of the prime number $p$. Moreover, a group $G$ is said to be {\it locally finite} if every finitely generated subgroup is finite.

\begin{lemma}\label{4.14} \cite[Lemma $2$]{26}.
	Let $p$ be a prime with $p\in J(R)$. If $G$ is a locally finite $p$-group, then $\varepsilon(RG) \subseteq J(RG)$.
\end{lemma}

\begin{proposition} 
(i) If $RG$ is a 2-\(\Delta U\) ring, then $R$ is also a 2-\(\Delta U\) ring.\par
(ii) If \( R \) is a 2-\(\Delta U\) ring and \( G \) is a locally finite \( p \)-group, where \( p \) is a prime number such that \( p \in J(R) \), then \( RG \) is a 2-\(\Delta U\) ring.
\end{proposition}
\begin{proof}
($i$) Assume $u \in U(R)$, then $u \in U(RG)$. Thus, $u^2=1+r$ where $r \in \Delta(RG)$. Since $r=1-u^2 \in R$, it suffices to show that $r \in \Delta(R)$, which is obvious because for any $v \in U(R) \subseteq U(RG)$, we have $v-r \in U(RG) \cap R \subseteq U(R)$. Therefore, $r \in \Delta(R)$.\par
($ii$) By Lemma \ref{4.14}, we have $\varepsilon(RG) \subseteq J(RG)$. On the other hand, since \( RG/\varepsilon(RG) \cong R \), by Theorem \ref{3.5}, we conclude that \( RG \) is a 2-\(\Delta U\) ring.
\end{proof}

\begin{proposition}
If the group ring \( RG \) is a 2-\(\Delta U\) ring with \( 2 \in \Delta(RG) \), then \( G \) is a 2-group.
\end{proposition}
\begin{proof}
We first consider two claims:
	
\textbf{Claim 1: } Every element \( g \in G \) has a finite order.
	
Assume there exists \( g \in G \) with infinite order. Since \( RG \) is a 2-\(\Delta U\) ring, we have \( 1-g^2 \in \Delta(RG) \). Given \( 2 \in \Delta(RG) \), we then have \( 1+g^2 \in \Delta(RG) \), implying \( 1+g+g^2 \in U(RG) \). Therefore, there exist integers \( n < m \) and \( a_i \) with \( a_n \neq 0 \neq a_m \) such that
	\[
	(1+g+g^2)\sum_{n}^{m} a_ig^i=1.
	\]
This leads to a contradiction, thus every element \( g \in G \) must have finite order.
	
\textbf{Claim 2: } For any \( g \in G \) and \( k \in \mathbb{N} \), we have \(\sum_{i=0}^{n=2k} g^i \in U(RG)\).
	
We will prove this for \( k=1, 2 \), and the rest follows similarly.
	
For \( k=1 \), for any \( g \in G \), we have \( 1-g^2 \in \Delta(RG) \). Since \( 2 \in \Delta(RG) \), we then have \( 1+g^2 \in \Delta(RG) \) and hence \( 1+g+g^2 \in U(RG) \).
	
For \( k=2 \), for any \( g \in G \), since \( g, g^2 \in U(RG) \), we have:

\( 1-g^2 \in \Delta(RG) \) and hence \( 1+g^2 \in \Delta(RG) \). Thus, \( g+g^3 \in \Delta(RG) \). On the other hand \( 1-g^4 \in \Delta(RG) \) and hence \( 1+g^4 \in \Delta(RG) \).

Since \(\Delta(RG)\) is closed under addition, it follows that \( 2+g+g^2+g^3+g^4 \in \Delta(RG) \), implying \( g+g^2+g^3+g^4 \in \Delta(RG) \) and thus \( 1+g+g^2+g^3+g^4 \in U(RG) \).
	
Continuing this process, we can show that \(\sum_{i=0}^{n=2k} g^i \in U(RG)\).
	
If \( g \in G \) has an order \( p \) that does not divide 2, then \( p \) is odd, hence \( p-1=2k \). Consequently, \(\sum_{i=0}^{2k} g^i \in U(RG)\), and since \( (1-g)(\sum_{i=0}^{2k} g^i)=0 \), we have \( 1-g=0 \), which is a contradiction. Thus, \( G \) must be a 2-group.
\end{proof}

\medskip

\section{Open Questions}

We finish our work with the following questions which allude us.

A ring $R$ is called $UQ$ if $U(R)=1+QN(R)$ (see \cite{21}).

\begin{problem}\label{5.1}
Examine those rings $R$ whose for each $u\in U(R)$, $u^2=1+q$ where $q\in QN(R)$ (i.e., 2-$UQ$ rings).
\end{problem}

\begin{problem}\label{5.2}
Characterize regular (or semi-regular) 2-$UQ$ rings.
\end{problem}

A ring $R$ is called $UNJ$ if $U(R)=1+Nil(R)+J(R)$ (see \cite{22}).

\begin{problem}\label{5.3}
Examine those rings $R$ whose for each $u\in U(R)$, $u^2=1+n+j,$ where $n\in Nil(R)$ and $j\in J(R)$ (i.e., 2-$UNJ$ rings).
\end{problem}

\begin{problem}\label{5.4}
Characterize regular (or semi-regular) 2-$UNJ$ rings.
\end{problem}

\medskip
\medskip


\medskip 

\noindent{\bf Funding:} The first and second authors are supported by Bonyad-Meli-Nokhbegan and receive funds from this foundation. The work of the third-named author, P.V. Danchev, is partially supported by the Junta de Andaluc\'ia, Grant FQM 264.

\end{document}